\documentclass[reqno, 12pt]{amsart}
\usepackage[letterpaper, margin=1in]{geometry}

\usepackage{amsmath, amssymb, amsfonts, amsthm, mathtools} 
\usepackage{enumitem} 
\setenumerate{label = \textup{(\textit{\roman*})}}
\usepackage[dvipsnames]{xcolor}
\usepackage{tikz}
\usetikzlibrary{calc,decorations.markings,arrows.meta}
\usepackage{tkz-graph}
\usepackage{comment}
\usepackage{bm}
\usepackage{upgreek}
\usepackage[english]{babel}
\usepackage{csquotes}
\usepackage{tikz-cd}
\usepackage{subcaption}
\usepackage{float}
\usepackage{hyperref}
\usepackage{microtype}

\allowdisplaybreaks[3]

\hypersetup{
    linkcolor = blue!80!gray,
    citecolor = blue!80!gray,
    urlcolor = blue!80!gray,
    colorlinks = true,
}
\usepackage[noabbrev, capitalize, nameinlink]{cleveref}

\newcommand{\colorA}{Orchid}
\newcommand{\colorB}{Tan}
\newcommand{\colorC}{JungleGreen}
\newcommand{\colorD}{NavyBlue}
\newcommand{\colorE}{gray}

\theoremstyle{definition}
\newtheorem{theorem}{Theorem}[section]
\newtheorem{definition}[theorem]{Definition}
\newtheorem{lemma}[theorem]{Lemma}
\newtheorem{proposition}[theorem]{Proposition}
\newtheorem{corollary}[theorem]{Corollary}
\newtheorem{conjecture}[theorem]{Conjecture}

\newtheorem{example}[theorem]{Example}

\newtheorem*{remark}{Remark}

\tikzset{
  mynode/.style={fill,circle,draw,inner sep=2pt,outer sep=0pt}
}

\newcommand{\R}{\mathbb{R}}

\newcommand{\ZZ}{\mathbb{Z}}
\newcommand{\ZZZ}{\mathbb{Z}_{\ge 0}}

\newcommand{\B}{\mathcal{B}}
\newcommand{\N}{\mathbb{N}}

\renewcommand{\P}{\mathcal{P}}

\newcommand{\id}{\mathrm{id}}

\newcommand{\Newton}{\mathrm{Newton}}
\newcommand{\conv}{\mathrm{conv}}
\newcommand{\Inv}{\mathrm{Inv}}
\newcommand{\GW}{\mathrm{GW}}

\newcommand{\supp}{\mathrm{supp}}

\usepackage{dsfont}
\newcommand{\1}{\mathds{1}}

\title{Newton polytopes of dual Schubert polynomials}
\date{\today}

\author{Serena An}
\address{Massachusetts Institute of Technology}
\email{anser@mit.edu}

\author{Katherine Tung}
\address{Harvard University}
\email{katherinetung@college.harvard.edu}

\author{Yuchong Zhang}
\address{University of Michigan}
\email{zongxun@umich.edu}

\begin{document}

\begin{abstract}
The M-convexity of dual Schubert polynomials was first proven by Huh, Matherne, M\'esz\'aros, and St.\ Dizier in 2022.
We give a full characterization of the supports of dual Schubert polynomials, which yields an elementary alternative proof of the M-convexity result, and furthermore strengthens it by explicitly characterizing the vertices of their Newton polytopes combinatorially.
\end{abstract}

\maketitle

\section{Introduction}
Dual Schubert polynomials $\{D^w\}_{w\in S_{n}}$, introduced by Bernstein, Gelfand, and Gelfand \cite{BGG}, are given by the following combinatorial formula \cite{PostnikovStanley}. Let the edge $u \lessdot ut_{ab}$ in the (strong) Bruhat order have weight 
\[m(u \lessdot ut_{ab}) \coloneqq x_a + x_{a+1} + \cdots + x_{b-1},\] 
and let the saturated chain $C = (u_0 \lessdot u_1\lessdot\cdots \lessdot u_{\ell})$ have weight 
\[m_C\coloneq m(u_0 \lessdot u_1)m(u_1 \lessdot u_2)\cdots m(u_{\ell - 1} \lessdot u_{\ell}).\]

\begin{definition}\label{def:dualSchubert}
    For $w \in S_n$, the \emph{dual Schubert polynomial $D^w$} is defined by \[D^w(x_1, \dots,x_{n-1}) \coloneqq \frac{1}{\ell(w)!} \sum_C m_C(x_1, \dots,x_{n-1}),\] where $\ell(w)$ denotes the Coxeter length of $w$, and the sum is over all saturated chains $C$ from $\id$ to $w$.
\end{definition}

Dual Schubert polynomials possess deep geometric and algebraic properties. For example, the $\lambda$-degree of a Schubert variety $X_w$ can be expressed as $\ell(w)! \cdot D^w(\lambda)$ \cite{PostnikovStanley}. Additionally, dual Schubert polynomials form a dual basis to Schubert polynomials under a certain natural pairing $\langle \cdot, \cdot \rangle$ on polynomials \cite{PostnikovStanley}. 

In recent years, dual Schubert polynomials have attracted increasing interest \cite{InvolRC, Hamaker,HuhMeszaroLorSchur}. We continue the study of dual Schubert polynomials by fully characterizing their \emph{supports} (set of exponent vectors) and \emph{Newton polytopes} (convex hull of the support).

\begin{theorem}\label{theorem:support}
    The support of the dual Schubert polynomial $D^w$ is
    \[\supp (D^w)= \sum_{(a,b) \in \Inv(w)}\{e_a, e_{a+1}, \dots, e_{b-1}\},\]
    where the right-hand side is a Minkowski sum of sets of elementary basis vectors. The sum is over pairs of indices $(a, b)$ for which there is an inversion in $w$.
\end{theorem}

A polynomial $f$ has a \emph{saturated Newton polytope (SNP)} if its support is all integer points in its Newton polytope. The SNP property was first defined by Monical, Tokcan and Yong \cite{NewtonPolyinAC}. 
Many polynomials with algebraic combinatorial significance are known to have SNP, such as Schur polynomials \cite{Rado} and resultants \cite{GKZ}. 
Monical, Tokcan, and Yong proved SNP for additional families of polynomials, including cycle index polynomials, Reutenauer’s symmetric polynomials linked to the free Lie algebra and to Witt vectors, Stembridge’s symmetric polynomials associated to totally nonnegative matrices, and symmetric Macdonald polynomials. 
Subsequent work of Fink, M\'esz\'aros, and St.\ Dizier proved SNP for key polynomials and Schubert polynomials \cite{FMD}, work of Castillo, Cid Ruiz, Mohammadi, and Monta\~no proved SNP for double Schubert polynomials \cite{CCMM}, and work of Huh, Matherne, M\'esz\'aros, and St.\ Dizier \cite{HuhMeszaroLorSchur} proved Lorentzian-ness, which implies SNP, for dual Schubert polynomials.

Proving SNP can be difficult given that many polynomial operations, such as multiplication, do not preserve SNP. However, a wide range of techniques have been harnessed to prove SNP. For instance, Rado uses elementary combinatorial techniques \cite{Rado}, Fink, M\'esz\'aros, and St.\ Dizier rely on representation theory \cite{FMD}, and Castillo et al.\ as well as Huh et al.\ use results from algebraic geometry \cite{CCMM,HuhMeszaroLorSchur}. 

As a corollary of \Cref{theorem:support}, we obtain an elementary proof of the following result.

\begin{corollary}(cf.\ \cite[Proposition~18]{HuhMeszaroLorSchur})\label{thm:M-conv}
Dual Schubert polynomials are \emph{M-convex}; equivalently, they have SNP and their Newton polytopes are generalized permutahedra.
\end{corollary}

Strengthening this result, we fully characterize the vertices of $\Newton(D^w)$.

\begin{corollary}\label{thm:DSPNPVert}
    For $w\in S_n$, the vertices of $\Newton(D^w)$ are
    \[\{\alpha\in\ZZ_{\ge 0}^{n-1} \mid x^{\alpha}\text{ has coefficient }1\text{ in }\prod_{(a,b) \in \Inv(w)} (x_a+x_{a+1}+\dots +x_{b-1})\}.\]
\end{corollary}

Furthermore, in light of \cite[Corollary~8.2]{postnikovpermutohedra}, we give a combinatorial characterization of the vertices of $\Newton(D^w)$ using rectangle tilings of staircase Young diagrams.

This paper is organized as follows. In \cref{section:preliminaries}, we give the necessary relevant background and definitions. In \cref{section:Mconvexity}, we prove that dual Schubert polynomials are M-convex. In \cref{section:vertices}, we characterize the vertices of Newton polytopes of dual Schubert polynomials in two different ways.

\section*{Acknowledgments}

This research was conducted through the University of Minnesota Combinatorics and Algebra REU and supported by NSF grant DMS-2053288 and the D.E.\ Shaw group. We are grateful to our mentor Shiyun Wang and our TA Meagan Kenney for their guidance, as well as Pavlo Pylyavskyy for his invaluable support throughout the project. We thank Ayah Almousa, Casey Appleton, Shiliang Gao, Yibo Gao, June Huh, Jonas Iskander, Victor Reiner, and Lauren Williams for helpful discussions. We also thank Grant Barkley, Meagan Kenney, Mitchell Lee, Pavlo Pylyavskyy, Victor Reiner, and Shiyun Wang for their feedback on earlier drafts of this paper.

\section{Preliminaries}\label{section:preliminaries}
\subsection{Bruhat order}
In this paper, we use standard terminology for the (strong) Bruhat order. The readers are referred to \cite[Chapter~2]{BjornerBrenti} for a more detailed exposition.

Let $S_n$ be the symmetric group on the set $[n] \coloneqq \{1, \dots, n\}$. 
We write each permutation $w\in S_n$ in one-line notation as $w = w(1)w(2)\cdots w(n)$. 

Each permutation $w \in S_n$ can be expressed as a product of \emph{simple transpositions} $s_i = \{(i, i+1) : 1\le i\le n-1\}$. If $w = s_{i_1}\cdots s_{i_{\ell}}$ is expressed as a product of simple transpositions with $\ell$ minimal among all such expressions, then the string $s_{i_1}\cdots s_{i_{\ell}}$ is called a \emph{reduced decomposition} of $w$. We call $\ell \coloneqq \ell(w)$ the \emph{(Coxeter) length} of $w$.
 
An \emph{inversion} of $w$ is an ordered pair $(a,b)\in [n]^2$ such that $a< b$ and $w(a)>w(b)$. We denote the set of all inversions of $w$ by $\Inv(w)$. It is well-known that $\ell(w) = |\Inv(w)|$. Given $1\le a < b\le n$, let $t_{ab}$ act on $w\in S_n$ such that $wt_{ab}$ is the permutation obtained by transposing the two numbers in positions $a$ and $b$ in $w$.

\begin{definition}(\cite{BjornerBrenti})\label{def:bruhatorder}
    We define a partial order $\le$ on $S_n$ called the \emph{strong Bruhat order} as follows. Let $u, v\in S_n$ and $\ell = \ell(v)$. 
    We have $u\le v$ if and only if for any reduced decomposition $s_{i_1}\cdots s_{i_{\ell}}$ of $v$, there exists a reduced decomposition $s_{j_1}\cdots s_{j_k}$ of $u$ such that $s_{j_1}\cdots s_{j_k}$ is a substring of $s_{i_1}\cdots s_{i_{\ell}}$. If additionally $\ell(v) = \ell(u) + 1$, we write $u\lessdot v$. Alternatively, we may characterize the covering relation as follows. For $u, v \in S_n$, we have $u \lessdot v$ if and only if $v = ut_{ab}$ and $\ell(v) = \ell(u) + 1$.
\end{definition}

For $u\le w$ in the Bruhat order of $S_n$, the \emph{(Bruhat) interval} $[u, w]$ is the subposet containing all $v\in S_n$ such that $u\le v\le w$. 

\subsection{Postnikov--Stanley polynomials}
We define Postnikov--Stanley polynomials, which generalize dual Schubert polynomials. 

\begin{definition}(\cite{PostnikovStanley})
    Given $u\le w $ in $ S_n$, the \emph{Postnikov--Stanley polynomial} $D_u^w$ is defined by
    \[D_u^w(x_1, \dots, x_{n-1}) \coloneqq \frac{1}{(\ell(w) - \ell(u))!}\sum_{C} m_C(x_1, \dots, x_{n-1}),\]
    where the sum is over all saturated chains $C$ from $u$ to $w$.
\end{definition}
As defined in the introduction, the \emph{dual Schubert polynomial} $D^w$ is given by $D^w\coloneq D_{\id}^w$.
\begin{example}\label{ex:PS}
    Let $u = 213$ and $w = 321$. In the Bruhat order of $S_3$, there are two saturated chains from $u$ to $w$, namely $213\lessdot 231\lessdot 321$ and $213\lessdot 312\lessdot 321$, as shown in \cref{fig:123to321}. The first chain has weight $x_1x_2$ and the second chain has weight $(x_1 + x_2) \cdot x_2$. Thus, $D_{213}^{321} = \frac{1}{2!}(x_1 x_2 + (x_1 + x_2)\cdot x_2)$. 
\end{example}

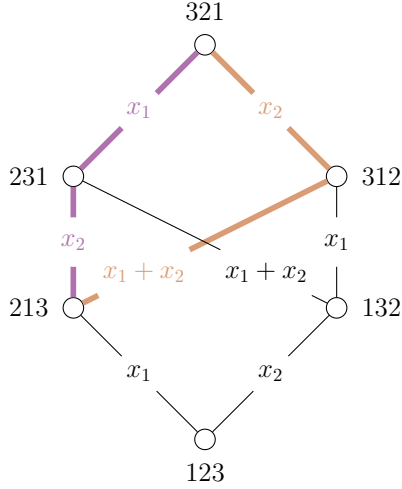
\begin{figure}
    \centering
    \begin{tikzpicture}[auto=center,every node/.style={circle, fill=white,scale=0.7,draw=black, solid}, scale=1.75, label distance=0.5mm]
        \node (a) at (0, 0) {};
        \node[label=left:{\large $213$}] (b) at (-1, 1) {};
        \node[label=right:{\large $132$}] (d) at (1, 1) {};
        \node[label=left:{\large $231$}] (e) at (-1, 2) {};
        \node[label=right:{\large $312$}] (g) at (1, 2) {};
        \node (h) at (0, 3) {};

        \node[draw=none, fill=none] at ($(a)-(0, 0.25)$) {\large $123$};
        \node[draw=none, fill=none] at ($(h)+(0, 0.25)$) {\large $321$};
        
        \draw (a)--node [draw=none] {\large $x_1$}(b);
        \draw (a)--node [draw=none] {\large $x_2$}(d);
        \draw[line width=0.75mm, \colorA] (b)--node [draw=none] {\large $x_2$}(e);
        \draw[line width=0.75mm, \colorA] (e)--node [draw=none] {\large $x_1$}(h);
        \draw[line width=0.75mm, \colorB] (g)--node [draw=none] {\large $x_2$}(h);
        \draw (d)--node [draw=none] {\large $x_1$}(g);
        \draw[line width=0.75mm, \colorB] (b)--node [near start, draw=none] {\large $x_1 + x_2$}(g);
        \draw (d)--node [near start, draw=none] {\large $x_1 + x_2$}(e);
    \end{tikzpicture}
    \caption{Calculating the Postnikov--Stanley polynomial $D_{213}^{321}$.}
    \label{fig:123to321}
\end{figure}

\subsection{Newton polytopes}
We rigorize the parenthetical definitions of the Newton polytope given in the introduction. 
\begin{definition}\label{def:expvector}
    For a tuple $\alpha = (\alpha_1, \dots, \alpha_n)\in \ZZ_{\ge 0}^n$, let $x^\alpha$ denote the monomial 
    \[x^\alpha \coloneqq x_1^{\alpha_1}\cdots x_n^{\alpha_n}\in \R[x_1, \dots, x_n].\]
    We call $\alpha$ the \emph{exponent vector} of $x^{\alpha}$.
\end{definition}
\begin{definition}\label{def:supp}
Let $f = \sum_{\alpha \in \ZZZ^{n}} c_{\alpha} x^{\alpha} \in \R[x_{1}, \ldots, x_{n}]$ be a polynomial. 
    The \emph{support} of $f$, denoted $\supp(f)$, is the set of its {exponent vectors} $\alpha$.
\end{definition}
\begin{definition}\label{def:newton}
    The \emph{Newton polytope} of a polynomial $f\in \R[x_1, \dots, x_n]$, denoted $\Newton(f)$, is the convex hull of $\supp(f)$ in $\mathbb{R}^n$.
\end{definition}

The support and Newton polytope of a polynomial behave nicely with respect to polynomial addition and multiplication as follows. 
\begin{proposition}\label{prop:suppsumprod}
    For two polynomials $f, g\in\R_{\ge 0}[x_1, \dots, x_n]$, we have 
    \begin{align*}
        \supp(fg) &=\supp(f)+\supp(g)\\
        \supp(f+g) &=\supp(f)\cup\supp(g),
    \end{align*}
     where $+$ in the first line denotes the Minkowski sum.
\end{proposition}

\begin{proposition}(\cite{NewtonPolyinAC})\label{prop:Newtonsumprod}
    For two polynomials $f, g\in\R_{\ge 0}[x_1, \dots, x_n]$, we have 
    \begin{align*}
        \Newton(fg) &=\Newton(f)+\Newton(g)\\
        \Newton(f+g) &=\conv(\Newton(f)\cup\Newton(g)),
    \end{align*}
     where $+$ in the first line denotes the Minkowski sum, and $\conv$ denotes the convex hull.
\end{proposition}
\begin{definition}(\cite{NewtonPolyinAC})
    A polynomial $f \in \R[x_1, \dots, x_n]$ has \emph{saturated Newton polytope (SNP)} if $\supp(f)$ consists of all integer points in $\Newton(f)$.    
\end{definition}
\begin{example}\label{ex:snp}
    As computed in \cref{ex:PS}, $D_{213}^{321} = x_1x_2+\frac{1}{2}x_2^2$, so its Newton polytope $\Newton(D_{213}^{321})$ is the line segment from $(1, 1)$ to $(0, 2)$ in $\R^2$. There are no integer points on this line segment besides the endpoints, so $D_{213}^{321}$ has SNP.
\end{example}

\begin{remark}
    A nonexample for SNP is the polynomial
    \[f = x^{(0, 1, 3)} + x^{(0,3,1)} + x^{(1,0,3)} + x^{(1,3,0)} + x^{(3,0,1)} + x^{(3,1,0)},\] 
    because $\Newton(f)$ contains the integer point $(0, 2, 2)$, but there is no $cx^{(0, 2, 2)}$ monomial in $f$ for $c$ nonzero.
\end{remark}

\subsection{Matroid polytopes}\label{section:matroidpolytopes}

Matroid polytopes, a special type of Newton polytope, were heavily used in \cite{DoubleSchSNP} to prove that double Schubert polynomials have SNP and characterize their Newton polytopes. We also leverage the following results about matroid polytopes to discuss products of linear polynomials in \cref{section:SCNP} and characterize the Newton polytope of dual Schubert polynomials in \cref{section:newtongeneralizedpermutahedra}.

\begin{definition}
    A \emph{matroid} $M = (E, \B)$ consists of a finite set $E$ and a nonempty collection of subsets $\B$ of $E$, called the \emph{bases} of $M$, which satisfy the \emph{basis exchange axiom}: if $B_1, B_2\in \B$ and $b_1\in B_1 \setminus B_2$, then there exists $b_2\in B_2 \setminus B_1$ such that $B_1 \setminus \{b_1\}\cup \{b_2\}\in \B$.
\end{definition}

In the remainder of this section, we let $E = [n]$.

\begin{definition}
    The \emph{matroid polytope} $P(M)$ of a matroid $M = ([n], \B)$ is 
    \[P(M) = \conv(\{\zeta^B : B\in \B\}),\]
    where $\zeta^B = (\1_{i\in B})_{i=1}^n$ denotes the indicator vector of $B$.
\end{definition}

Generalized permutahedra were first studied in \cite{postnikovpermutohedra}, and many nice connections to matroid polytopes have been found since.

A \emph{standard permutahedron} (or \emph{permutohedron}) is the convex hull in $\R^n$ of the vector $(0, 1, \dots, n-1)$ and all permutations of its entries. \emph{Generalized permutahedra}, which are deformations of standard permutahedra, are defined as follows.

\begin{definition}
    The \emph{generalized permutahedron} $P_n^z(\{z_I\})$ associated to the collection of real numbers $\{z_I\}$ for $I\subseteq [n]$, is given by
    \[P_n^z(\{z_I\}) = \left\{t\in\R^n : \sum_{i\in I} t_i\ge z_I \text{ for } I\neq [n], \sum_{i=1}^n t_i = z_{[n]}\right\}.\]
\end{definition}

\begin{proposition}(\cite{MatroidPolytopes})\label{prop:permutahedraminkowskisum}
    Generalized permutahedra are closed under Minkowski sums:
    \[P_n^z(\{z_I\}) + P_n^z(\{z_I'\}) = P_n^z(\{z_I + z_I'\}).\]
\end{proposition}

\begin{definition}\label{def:rankfunction}
    For a matroid $M = ([n], \B)$, the \emph{rank function} $r_M: 2^{[n]}\to\ZZ_{\ge 0}$ on subsets of $[n]$ is given by
    \[r_M(S) = \max\{\#(S\cap B) : B\in\B\}.\]
\end{definition}

\begin{proposition}(\cite{MatroidPolytopes})\label{prop:matroidsarepermutahedra}
    Matroid polytopes are generalized permutahedra with
    \begin{align*}
        P(M) &= P_n^z(\{r_M([n]) - r_M([n]\setminus I)\}_{I\subseteq [n]})\\
        &= \left\{t\in\R^n : \sum_{i\in I} t_i\le r_M(I) \text{ for } I\neq [n], \sum_{i=1}^n t_i = r_M([n])\right\}.
    \end{align*}
\end{proposition}

\begin{definition}(\cite{DicreteConvexAnaly})
A subset $\mathcal{J} \subset \ZZ^n$ is \emph{M-convex} if for any index $i \in [n]$ and any $\alpha, \beta \in \mathcal{J}$ whose $i$th coordinates satisfy $\alpha_i > \beta_i$, there is an index $j \in [n]$ satisfying
\[
\alpha_j < \beta_j \quad \text{and} \quad \alpha - e_i + e_j \in \mathcal{J} \quad \text{and} \quad \beta - e_j + e_i \in \mathcal{J}.
\]
    A polynomial is \emph{M-convex} if $\supp(f)$ is an M-convex set.
\end{definition}
\begin{proposition}(\cite[Theorem 4.15]{DicreteConvexAnaly}, \cite{HuhMeszaroLorSchur})\label{prop:Mconvex=SNP+GP}
    A homogeneous polynomial $f$ is M-convex if and only if $f$ has SNP and $\Newton(f)$ is a generalized permutahedron.
\end{proposition}

\section{M-convexity of dual Schubert polynomials}\label{section:Mconvexity}
In this section, we prove \Cref{theorem:support} and \Cref{thm:M-conv}. Our key insight is that the support of any dual Schubert polynomial $D^w$ always matches the support of the weight of a specific chain in $[\id,w]$. After showing that the weight of any chain has SNP and has a Newton polytope that is a generalized permutahedron, we obtain the desired result.
\subsection{The single chain Newton polytope (SCNP) property}\label{section:SCNP}
We introduce the {single chain Newton polytope} property and show that it implies, but is not equivalent to, the SNP property.

\begin{definition} \label{SCNP}
    Given $u \le w$ in $S_n$, the Postnikov--Stanley polynomial $D_{u}^w$ has \emph{single chain Newton polytope (SCNP)} if there exists a saturated chain $C$ in the interval $[u, w]$ such that
    \[\supp(m_C) = \supp(D_u^w).\]
    Such a saturated chain $C$ is called a \emph{dominant chain} of the interval $[u, w]$.
\end{definition}
\begin{example}
    Given the Postnikov--Stanley polynomial $D_{213}^{321} = \frac{1}{2!}(x_1 x_2 + (x_1 + x_2)\cdot x_2)$, the saturated chain $C = (213 \lessdot 312 \lessdot 321)$ has weight $m_C = (x_1 + x_2)\cdot x_2$, which satisfies 
    \[\supp(m_C) = \supp(D_{213}^{321}).\] 
    Thus, $C$ is a dominant chain of $[213, 321]$, and $D_{213}^{321}$ has SCNP. 
\end{example}

\begin{remark}\label{example:13244231}
    The SCNP property is stronger than the SNP property: for example, $D_{1324}^{4231}$ has SNP but not SCNP, according to calculations with SageMath \cite{sagemath}. As Postnikov--Stanley polynomials do not necessarily have SCNP, the method of using SCNP to prove SNP for dual Schubert polynomials does not generalize to Postnikov--Stanley polynomials. 
\end{remark}

The following propositions motivate the definition of SCNP.

\begin{proposition}\label{prop:linearprodSNP}
    If a polynomial $f\in \R[x_1,\dots,x_n]$ can be written as a product of nonnegative linear combinations of $x_1, \dots, x_n$, then $f$ has SNP.
\end{proposition}

\begin{proof}
    This follows from \cite[Theorem 3.4]{2009minkowski}.
\end{proof}
\begin{proposition}\label{prop:SCNPtoSNP}
    If $D_u^w$ has SCNP, then $D_u^w$ has SNP.
\end{proposition}
\begin{proof}
    By \cref{prop:linearprodSNP}, the chain weight $m_C$ has SNP for any saturated chain $C$ in the Bruhat order. If, by the definition of SCNP, there exists a saturated chain $C$ in the interval $[u, w]$ such that $\supp(m_C) = \supp(D_u^w)$, then $D_u^w$ also has SNP.
\end{proof}

\subsection{Using SCNP to prove SNP}\label{subsec:dualschubSCNP}

In this section, we prove that dual Schubert polynomials have SCNP. As a corollary, we obtain \Cref{theorem:support} and the result that dual Schubert polynomials have SNP. 

\begin{definition}
    For a Bruhat interval $[u,w]$, a saturated chain 
    \[u = w_0\lessdot w_1\lessdot w_2\lessdot \dots \lessdot w_{\ell} = w\] 
    from $u$ to $w$ is a \emph{greedy chain} if it satisfies the following condition for all $i\in [\ell]$: writing $w_{i}=w_{i-1}t_{ab}$ for $a < b$, there does not exist $w'_{i-1}\lessdot w_{i}$ with $w'_{i-1} \in [u,w]$ such that 
    \begin{enumerate}[label=\textnormal{(\roman*)}]
        \item $w_{i}=w'_{i-1}t_{ab'}$ for $b'>b$, or
        \item $w_{i}=w'_{i-1}t_{a'b}$ for $a'<a$.
    \end{enumerate}
\end{definition}

\begin{example}
    In the Bruhat interval $[123, 321]$, the saturated chain $123 \lessdot 132\lessdot 231 \lessdot 321$ is greedy, while $123 \lessdot 213 \lessdot 231 \lessdot 321$ is not. This chain fails to be greedy because \[w_2 = 231 = 213t_{23} = w_{1}t_{23}\] but $132 \lessdot w_2$ with $231 = 132t_{13}$, violating condition (ii). In general, greedy chains are not unique. For example, $123 \lessdot 213 \lessdot 312 \lessdot 321$ is another greedy chain in $[123,321]$.
\end{example}

\begin{lemma}\label{lem:existsgreedychain}
    There exists a greedy chain in every Bruhat interval $[u, w]$. 
\end{lemma}

\begin{proof}
    We construct a greedy chain $u = u_{\ell} \lessdot u_{\ell-1} \lessdot\cdots\lessdot u_{0} = w$ inductively, from the top down. Given a greedy chain $ u_{i-1} \lessdot \cdots\lessdot u_{0} = w$ for $i\in [\ell]$, set $u_{i}\coloneq u_{i-1}t_{ab}$ such that $(a, b)$ cannot be replaced by some longer $(a', b)$ or $(a, b')$. 
\end{proof}

\begin{definition}\label{def:globalweight}
    For a permutation $w\in S_n$, the \emph{global weight} $\GW(w)$ of $w$ is the polynomial
    \[\GW(w) = \prod_{(a,b) \in \Inv(w)} (x_a+x_{a+1}+\dots +x_{b-1}).\]
\end{definition}

\begin{lemma}\label{lemma:greedy=GW}
    Given $w\in S_n$, the weight of every greedy chain in $[\id,w]$ is $\GW(w)$.
\end{lemma}

\begin{proof}
   We induct on $\ell(w)$, with the base case of $\ell(w) = 1$ clear. Suppose we have proven the statement for all $w'\in S_n$ with $\ell(w') < \ell$, and let $\ell(w) = \ell$. Let 
    \[C=(\id=w_0\lessdot w_1\lessdot w_2\lessdot \dots \lessdot w_{\ell} =w)\] 
    be a greedy chain in $[\id,w]$, and suppose $w=w_{\ell-1}t_{ab}$ for $a < b$. Since
    \[C'=(\id=w_0\lessdot w_1\lessdot w_2\lessdot \dots \lessdot w_{\ell-1})\] 
    is a greedy chain in $[\id, w_{\ell-1}]$, it suffices to show by the inductive hypothesis that
    \[\GW(w)=\GW(w_{\ell-1})\cdot (x_{a}+x_{a+1}+\dots+x_{b-1}).\]
    
    We compare $\Inv(w)$ and $\Inv(w_{\ell-1})$. After swapping $w(a)$ and $w(b)$ in $w$, the pair $(a, b)$ is no longer in $\Inv(w_{\ell-1})$. It suffices to show that $\Inv(w_{\ell-1})=\Inv(w)\setminus \{(a,b)\}$. Note that every pair $(c,d)\in \Inv(w)$ satisfying $c\neq a, b$ and $d\neq a, b$ is in $\Inv(w_{\ell-1})$. The remaining pairs in $\Inv(w)$ are of one of the following forms:
    \begin{enumerate}[label=\textnormal{(\roman*)}]
        \item $(a,j)$ for $a<j<b$
        \item $(a,j)$ for $a<b<j$
        \item $(j,b)$ for $a<j<b$
        \item $(j,b)$ for $j<a<b$.
    \end{enumerate}

    Since $\ell(wt_{ab})=\ell(w)-1$, we have $w(j)\notin [w(b),w(a)]$ for every $j\in [a+1,b-1]$. Thus, all pairs under cases (i) and (iii) are in $\Inv(w_{\ell-1})$.

    For case (ii), suppose $(a,j)\in \Inv(w)\setminus \Inv(w_{\ell-1})$, so $w(j)\in [w(b),w(a)]$. Let $r > b$ be the smallest integer such that $w(r)\in [w(b), w(a)]$; note that such an $r$ exists because $j > b$ has this property. Then $\ell(wt_{ar})=\ell(w)-1$, which contradicts the fact that $C$ is a greedy chain. Case (iv) is similar to case (ii).

    In conclusion, we have $\Inv(w_{\ell-1})=\Inv(w)\setminus \{(a,b)\}$, which implies
    \[\GW(w)=\GW(w_{\ell-1})\cdot (x_{a}+x_{a+1}+\dots+x_{b-1}),\] as desired.
\end{proof}

\begin{definition}\label{def:dominant pairing}
    We define a partial order $\preceq$ on $\{(a,b)\in \N^2\mid a<b\}$ such that $(a,b)\preceq (c,d)$ if and only if $[a,b]\subseteq [c,d]$. 

    Given an integer $\ell$, we define a partial order $\preceq_{\ell}$ on multisets of $\{(a,b)\in \N^2\mid a<b\}$ with $\ell$ elements as follows: for two multisets $G$ and $H$, we say $G \preceq_{\ell} H$ if and only if there exists a pairing of elements in $G$ and $H$ such that for each pair $((a_i,b_i),(c_i,d_i))\in G\times H$ in this pairing, we have $(a_i,b_i)\preceq (c_i,d_i)$. Such a pairing is called a \emph{dominant pairing}.
\end{definition}

\begin{definition}
    For a saturated chain $C = (u_0 \lessdot u_1\lessdot\cdots \lessdot u_{\ell})$ in the Bruhat interval $[u_0,u_{\ell}]$ in $S_n$,
    we define its \emph{generating set} $G_C$ to be the multiset containing the pairs of the positions $(a_i, b_i)$ swapped along edges in $C$:
    \[G_C = \{(a_i,b_i)\in [n]\mid u_{i}=u_{i-1}t_{a_ib_i}, a_i < b_i, i\in[\ell]\}.\]
\end{definition}
\begin{example}\label{ex:Gc}
    For the saturated chain $123 \lessdot 132 \lessdot 231 \lessdot 321$, the generating set is $\{(2,3), (1,3), (1,2)\}$.
\end{example}

\begin{lemma}\label{lemma:chainorder}
    Given $w\in S_n$ with $\ell(w) = \ell$, for every saturated chain \[C=(\id=w_0\lessdot w_1\lessdot w_2\lessdot \dots \lessdot w_{\ell} =w)\] in $[\id, w]$, we have
    $G_C\preceq_{\ell} \Inv(w)$.
\end{lemma}

\begin{proof}
    We induct on $\ell$, with the base case of $\ell(w) = 1$ clear. Suppose we have proven the statement for all $w'\in S_n$ with $\ell(w') < \ell$, and let $\ell(w) = \ell$. Let $a < b$ satisfy $w=w_{\ell-1}t_{ab}$, and let 
    \[C'=(\id=w_0\lessdot w_1\lessdot w_2\lessdot \dots \lessdot w_{\ell-1}).\] 
    Since $G_C=G_{C'}\cup\{(a,b)\}$, and we know $G_{C'}\preceq_{\ell-1} \Inv(w_{\ell-1})$ from the inductive hypothesis, we have
    \[G_C=G_{C'}\cup\{(a,b)\}\preceq_{\ell}\Inv(w_{\ell-1})\cup\{(a,b)\}.\]
    It now suffices to show that $\Inv(w_{\ell-1})\cup\{(a,b)\}\preceq_{\ell} \Inv(w)$, which we do by constructing a dominant pairing $\P$.
     
    The multisets $\Inv(w_{\ell-1})\cup\{(a,b)\}$ and $\Inv(w)$ are in fact both sets, since $(a,b)\notin \Inv(w_{\ell-1})$.
    Note that every pair $(c, d)\in \Inv(w)$ satisfying $c\neq a, b$ and $d\neq a, b$ is in $\Inv(w_{\ell-1})$, and $(a,b)$ is in $\Inv(w_{\ell-1})\cup\{(a,b)\}$. The remaining pairs in $\Inv(w)$ are of one of the following forms:
    \begin{enumerate}[label=\textnormal{(\roman*)}]
        \item $(a,j)$ for $a<j<b$
        \item $(a,j)$ for $a<b<j$
        \item $(j,b)$ for $a<j<b$
        \item $(j,b)$ for $j<a<b$.
    \end{enumerate}

    Since $\ell(wt_{ab})=\ell(w)-1$, we have $w(j)\notin [w(b),w(a)]$ for every $j\in [a+1,b-1]$. Then the pairs from cases (i) and (iii) are also in $\Inv(w_{\ell-1})$.
    We put these pairs appearing in both $\Inv(w_{\ell-1})\cup\{(a,b)\}$ and $\Inv(w)$ together in $\P$. 

    Now suppose there is a pair $(a, j)\in\Inv(w)$ falling under case (ii). If $(b,j)$ is also in $\Inv(w)$, then $w(j) < w(b) < w(a)$. Hence $(a,j)$ is in $\Inv(w_{\ell-1})$, and we let $((a,j),(a,j))\in \P$. Otherwise if $(b,j)\notin \Inv(w)$, then $w(b) < w(j) < w(a)$. So $(b, j)\in \Inv(w_{\ell-1})\setminus \Inv(w)$, and we can let $((b,j),(a,j))\in \P$ since $[b,j]\subseteq[a,j]$. Case (iv) is similar to case (ii).
\end{proof}

\begin{lemma}\label{lemma:chainNPcontainment}
    Given $w\in S_n$, for every saturated chain $C$ in $[\id, w]$, we have
    \[\supp(m_C)\subseteq \supp(\GW(w)).\]
\end{lemma}

\begin{proof}
    Let $\ell = \ell(w)$. By \cref{lemma:chainorder}, we have $G_C\preceq_{\ell} \Inv(w)$, so there exists a dominant pairing $\P$ of $G_C$ and $\Inv(w)$ such that $[a,b]\subseteq [c,d]$ for each pair $((a,b),(c,d))\in \P$. Observe that each pair $(a,b)\in G_C$ corresponds to a linear factor $x_a+x_{a+1}+\dots+x_{b-1}$
    of $m_C$, and likewise each pair $(c,d)\in \Inv(w)$ corresponds to a linear factor $x_c+x_{c+1}+\dots+x_{d-1}$
    of $\GW(w)$. As a result, each monomial of $m_C$ is also in $\GW(w)$, so the lemma is shown.
\end{proof}

\begin{theorem}\label{theorem:DualSchubSCNP}
    Every greedy chain of $[\id, w]$ is a dominant chain of $D^w$, and 
    \[\supp (D^w)= \supp(\GW(w)).\]
    As a result, for all $w\in S_n$, the dual Schubert polynomial $D^w$ has SCNP.
\end{theorem}

\begin{proof}
The inclusion $\subseteq$ follows from \cref{lemma:chainNPcontainment} and \cref{prop:suppsumprod}. The reverse inclusion $\supseteq$ follows from \cref{lemma:greedy=GW} and \cref{lem:existsgreedychain}.
\end{proof}

\begin{corollary}\label{cor:DualSchubSNP}
    For all $w\in S_n$, the dual Schubert polynomial $D^w$ has SNP.
\end{corollary}

\begin{proof}
    This follows from \cref{prop:SCNPtoSNP} and \cref{theorem:DualSchubSCNP}. 
\end{proof}

Moreover, \Cref{theorem:DualSchubSCNP} gives a more precise characterization of $\supp(D^w)$, as stated in \Cref{theorem:support}.

\begin{proof}[Proof of \Cref{theorem:support}]
    This follows from \Cref{theorem:DualSchubSCNP}, \Cref{def:globalweight}, and \Cref{prop:suppsumprod}.
\end{proof}

\subsection{Newton polytopes as generalized permutahedra}\label{section:newtongeneralizedpermutahedra}

\begin{definition}
    For $1\le a < b\le n$, let $M_{ab} = ([n], \B)$ be the matroid on $[n]$ with bases $\B = \{\{a\}, \{a + 1\}, \dots, \{b-1\}\}$.
\end{definition}
The motivation for defining such a matroid is the following observation about its matroid polytope:
letting $e_i\in\R^n$ denote the unit vector with a 1 in the $i$th coordinate, we have 
\begin{align}\label{eq:motivation}
    P(M_{ab}) = \conv\{e_a, e_{a+1}, \dots, e_{b-1}\} = \Newton(x_a + x_{a+1} + \cdots + x_{b-1}).
\end{align}

Now we may apply theorems about matroid polytopes from \cref{section:matroidpolytopes} to obtain a characterization of $\Newton(D^w)$ as a generalized permutahedron.
In the following, let $I\supseteq [a, b)$ denote $I\supseteq \{a, a+1, \dots, b-1\}$, for $I\subseteq [n]$. 

\begin{theorem}\label{thm:DualSchubertAsGeneralizedPermutahedron}
    For $w\in S_n$, $\Newton(D^w)$ is a generalized permutahedron $P_n^z(\{z_I\})_{I\subseteq [n]}$ with
    \[z_I = \sum_{(a, b)\in \Inv(w)}\1_{I\supseteq[a, b)}\]
    for all $I\subseteq [n]$.
\end{theorem}

\begin{proof}

    By \cref{def:rankfunction}, we have for $I\subseteq [n]$ that
    \[r_{M_{ab}}([n]\setminus I) = \begin{cases}
        0 & \text{if } I\supseteq [a, b)\\
        1 & \text{if } I\nsupseteq [a, b)
    \end{cases}.\]
    Then from \cref{prop:matroidsarepermutahedra}, we have 
    \begin{align*}
        P(M_{ab}) &= P_n^z(\{r_{M_{ab}}([n]) - r_{M_{ab}}([n]\setminus I)\})_{I\subseteq [n]}\\
        &= P_n^z(\{1 - r_{M_{ab}}([n]\setminus I)\})_{I\subseteq [n]}\\
        &= P_n^z(\{\1_{I\supseteq [a, b)}\})_{I\subseteq [n]}.
    \end{align*}

    Recall from \cref{theorem:DualSchubSCNP} and \cref{def:globalweight} that $\Newton(D^w) = \Newton(\GW(w))$.
    We may then write 
    \begin{align*}
        \Newton(D^w) &= \sum_{(a, b)\in\Inv(w)} \Newton (x_a + x_{a+1} + \cdots + x_{b-1}) &\text{\cref{prop:Newtonsumprod}}\\
        &= \sum_{(a, b)\in\Inv(w)} P(M_{ab}) &\text{\cref{eq:motivation}} \\
        &= \sum_{(a, b)\in\Inv(w)} P_n^z(\{\1_{I\supseteq [a, b)}\})_{I\subseteq [n]} &\text{ \cref{def:rankfunction}, \cref{prop:matroidsarepermutahedra}} \\
        &= P_n^z\Big(\Big\{\sum_{(a, b)\in \Inv(w)}\1_{I\supseteq[a, b)}\Big\}\Big)_{I\subseteq [n]}. & \text{\cref{prop:permutahedraminkowskisum}}.
    \end{align*}
\end{proof}

Now we are ready to prove \Cref{thm:M-conv}. 
\begin{proof}[Proof of \Cref{thm:M-conv}]
    This follows from \Cref{cor:DualSchubSNP}, \Cref{thm:DualSchubertAsGeneralizedPermutahedron} and \Cref{prop:Mconvex=SNP+GP}.
\end{proof}

\section{Vertices of Newton polytopes}\label{section:vertices}

We characterize the vertices of $\Newton(D^w)$, as described in \Cref{thm:DSPNPVert}.

\begin{proof}[Proof of \Cref{thm:DSPNPVert}]
    This follows from \Cref{theorem:DualSchubSCNP} and \cite[Theorem 3.5]{2009minkowski}.
\end{proof}

Furthermore, we describe a combinatorial procedure to obtain the vertices of $\Newton(D^w)$ for a given $w \in S_n$. 

\begin{itemize}
    \item \textit{Step 1:} Construct a Young diagram of staircase shape $(n-1, n - 2, \ldots, 1)$, and label the boxes by the following pairs of inversions: in the $i$th row of the diagram for $1\le i \le n-1$, label the boxes from left to right by $(i,n), (i,n-1), \dots, (i,i+1)$.
    \item \textit{Step 2:} In each box, write a 1 if the inversion pair is in $\Inv(w)$, and a 0 otherwise.
    \item \textit{Step 3:} Construct the $\frac{1}{n+1}\binom{2n}{n}$ tilings of the Young diagram by $n-1$ rectangles. 
    \item \textit{Step 4:} For each tiling, sum the entries of each rectangle and write the sum at the bottom right corner of the rectangle. Reading the summands from top to bottom gives a vertex of $\Newton(D^w)$. 
\end{itemize}

\begin{corollary}
    The procedure above gives all the vertices of $\Newton(D^w)$.
\end{corollary}
    
\begin{proof}
    This follows from \Cref{theorem:support} and \cite[Corollary~8.2]{postnikovpermutohedra}.
\end{proof}

\begin{example}
    \Cref{fig:01061} illustrates the procedure for obtaining the vertex $(0,1,0,6,1)$ of $\Newton(D^{254361})$.
\begin{figure}[H]
\centering
\begin{minipage}{0.45\textwidth}
\centering
\begin{tikzpicture}[scale=1.2]

    \draw (0.35, -0.2) node {\scalebox{0.8}{$(1,6)$}};
    \draw (1.35, -0.2) node {\scalebox{0.8}{$(1,5)$}};
    \draw (2.35, -0.2) node {\scalebox{0.8}{$(1,4)$}};
     \draw (3.35, -0.2) node {\scalebox{0.8}{$(1,3)$}};
    \draw (4.35, -0.2) node {\scalebox{0.8}{$(1,2)$}};
    \draw (0.35, -1.2) node {\scalebox{0.8}{$(2,6)$}};
    \draw (1.35, -1.2) node {\scalebox{0.8}{$(2,5)$}};
    \draw (2.35, -1.2) node {\scalebox{0.8}{$(2,4)$}};
    \draw (3.35, -1.2) node {\scalebox{0.8}{$(2,3)$}};
    \draw (0.35, -2.2) node {\scalebox{0.8}{$(3,6)$}};
    \draw (1.35, -2.2) node {\scalebox{0.8}{$(3,5)$}};
    \draw (2.35, -2.2) node {\scalebox{0.8}{$(3,4)$}};
    \draw (0.35, -3.2) node {\scalebox{0.8}{$(4,6)$}};
    \draw (1.35, -3.2) node {\scalebox{0.8}{$(4,5)$}};
    \draw (0.35, -4.2) node {\scalebox{0.8}{$(5,6)$}};

    \draw[thick] (0, -1) -- (4, -1);
    \draw[thick] (0, -2) -- (3, -2);
    \draw[thick] (0, -3) -- (2, -3);
    \draw[thick] (0, -4) -- (1, -4);
    
    \draw[thick] (1, 0) -- (1, -4);
    \draw[thick] (2, 0) -- (2, -3);
    \draw[thick] (3, 0) -- (3, -2);
    \draw[thick] (4, 0) -- (4, -1);

    \draw[thick] (0, 0) -- (5, 0) -- (5, -1) -- (4, -1) -- (4, -2) -- (3, -2) -- (3, -3) -- (2, -3) -- (2,-4) -- (1,-4) -- (1, -5) -- (0, -5) -- (0,0);
\end{tikzpicture}
\caption*{\textit{Step 1:} A staircase Young diagram with $n= 6$.}
\label{fig:step1}
    
\end{minipage}
\hfill
\begin{minipage}{0.45\textwidth}
\centering
    \begin{tikzpicture}[scale=1.2]

    \draw (0.35, -0.2) node {\scalebox{0.8}{$(1,6)$}};
    \draw (1.35, -0.2) node {\scalebox{0.8}{$(1,5)$}};
    \draw (2.35, -0.2) node {\scalebox{0.8}{$(1,4)$}};
     \draw (3.35, -0.2) node {\scalebox{0.8}{$(1,3)$}};
    \draw (4.35, -0.2) node {\scalebox{0.8}{$(1,2)$}};
    \draw (0.35, -1.2) node {\scalebox{0.8}{$(2,6)$}};
    \draw (1.35, -1.2) node {\scalebox{0.8}{$(2,5)$}};
    \draw (2.35, -1.2) node {\scalebox{0.8}{$(2,4)$}};
    \draw (3.35, -1.2) node {\scalebox{0.8}{$(2,3)$}};
    \draw (0.35, -2.2) node {\scalebox{0.8}{$(3,6)$}};
    \draw (1.35, -2.2) node {\scalebox{0.8}{$(3,5)$}};
    \draw (2.35, -2.2) node {\scalebox{0.8}{$(3,4)$}};
    \draw (0.35, -3.2) node {\scalebox{0.8}{$(4,6)$}};
    \draw (1.35, -3.2) node {\scalebox{0.8}{$(4,5)$}};
    \draw (0.35, -4.2) node {\scalebox{0.8}{$(5,6)$}};
    \draw (0.5, -0.5) node {1};
    \draw (1.5, -0.5) node {0};
    \draw (2.5, -0.5) node {0};
    \draw (3.5, -0.5) node {0};
    \draw (4.5, -0.5) node {0};
    \draw (0.5, -1.5) node {1};
    \draw (1.5, -1.5) node {1};
    \draw (2.5, -1.5) node {0};
    \draw (3.5, -1.5) node {1};
    \draw (0.5, -2.5) node {1};
    \draw (1.5, -2.5) node {0};
    \draw (2.5, -2.5) node {0};
    \draw (0.5, -3.5) node {1};
    \draw (1.5, -3.5) node {1};
    \draw (0.5, -4.5) node {1};

    \draw[thick] (0, -1) -- (4, -1);
    \draw[thick] (0, -2) -- (3, -2);
    \draw[thick] (0, -3) -- (2, -3);
    \draw[thick] (0, -4) -- (1, -4);
    
    \draw[thick] (1, 0) -- (1, -4);
    \draw[thick] (2, 0) -- (2, -3);
    \draw[thick] (3, 0) -- (3, -2);
    \draw[thick] (4, 0) -- (4, -1);

    \draw[thick] (0, 0) -- (5, 0) -- (5, -1) -- (4, -1) -- (4, -2) -- (3, -2) -- (3, -3) -- (2, -3) -- (2,-4) -- (1,-4) -- (1, -5) -- (0, -5) -- (0,0);
\end{tikzpicture}
\caption*{\textit{Step 2:} When $w = 253641$, the boxes are filled with $1$'s as shown.}
\label{fig:step2}
\end{minipage}

\end{figure}

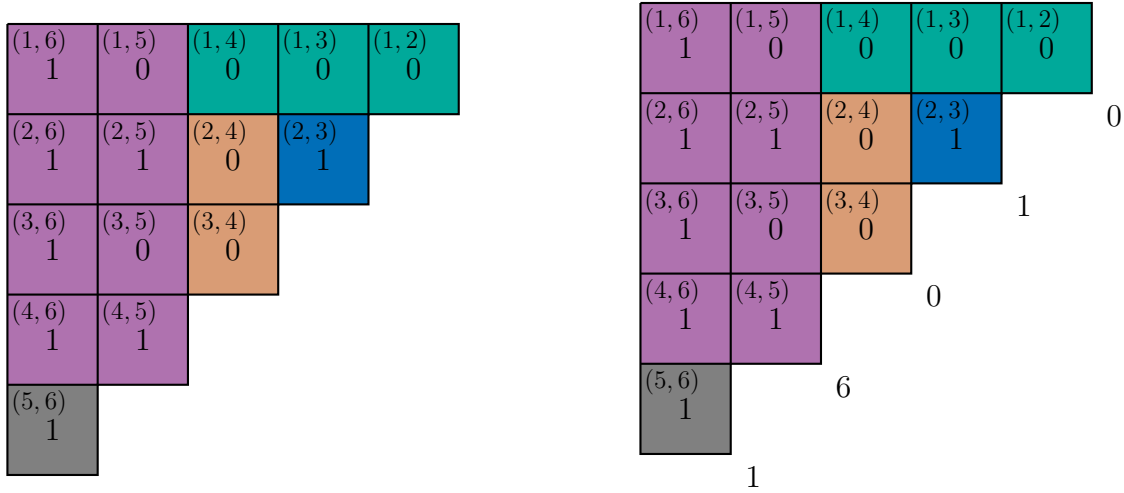
\begin{figure}[H]
\centering
\begin{minipage}{0.45\textwidth}
\centering
\begin{tikzpicture}[scale=1.2]
    \fill[\colorA, nearly transparent] (0, 0) -- (2, 0) -- (2, -4) -- (0, -4);
    \fill[\colorB, nearly transparent] (2,-1) -- (2, -3) -- (3, -3) -- (3, -1);
    \fill[\colorC, nearly transparent] (2, 0) -- (5, 0) -- (5, -1) -- (2, -1);
    \fill[\colorD, nearly transparent] (3, -1) -- (4, -1) -- (4, -2) -- (3, -2);
     \fill[\colorE, nearly transparent] (0, -4) -- (1, -4) -- (1, -5) -- (0, -5);

    \draw (0.35, -0.2) node {\scalebox{0.8}{$(1,6)$}};
    \draw (1.35, -0.2) node {\scalebox{0.8}{$(1,5)$}};
    \draw (2.35, -0.2) node {\scalebox{0.8}{$(1,4)$}};
     \draw (3.35, -0.2) node {\scalebox{0.8}{$(1,3)$}};
    \draw (4.35, -0.2) node {\scalebox{0.8}{$(1,2)$}};
    \draw (0.35, -1.2) node {\scalebox{0.8}{$(2,6)$}};
    \draw (1.35, -1.2) node {\scalebox{0.8}{$(2,5)$}};
    \draw (2.35, -1.2) node {\scalebox{0.8}{$(2,4)$}};
    \draw (3.35, -1.2) node {\scalebox{0.8}{$(2,3)$}};
    \draw (0.35, -2.2) node {\scalebox{0.8}{$(3,6)$}};
    \draw (1.35, -2.2) node {\scalebox{0.8}{$(3,5)$}};
    \draw (2.35, -2.2) node {\scalebox{0.8}{$(3,4)$}};
    \draw (0.35, -3.2) node {\scalebox{0.8}{$(4,6)$}};
    \draw (1.35, -3.2) node {\scalebox{0.8}{$(4,5)$}};
    \draw (0.35, -4.2) node {\scalebox{0.8}{$(5,6)$}};
    \draw (0.5, -0.5) node {1};
    \draw (1.5, -0.5) node {0};
    \draw (2.5, -0.5) node {0};
    \draw (3.5, -0.5) node {0};
    \draw (4.5, -0.5) node {0};
    \draw (0.5, -1.5) node {1};
    \draw (1.5, -1.5) node {1};
    \draw (2.5, -1.5) node {0};
    \draw (3.5, -1.5) node {1};
    \draw (0.5, -2.5) node {1};
    \draw (1.5, -2.5) node {0};
    \draw (2.5, -2.5) node {0};
    \draw (0.5, -3.5) node {1};
    \draw (1.5, -3.5) node {1};
    \draw (0.5, -4.5) node {1};

    \draw[thick] (0, -1) -- (4, -1);
    \draw[thick] (0, -2) -- (3, -2);
    \draw[thick] (0, -3) -- (2, -3);
    \draw[thick] (0, -4) -- (1, -4);
    
    \draw[thick] (1, 0) -- (1, -4);
    \draw[thick] (2, 0) -- (2, -3);
    \draw[thick] (3, 0) -- (3, -2);
    \draw[thick] (4, 0) -- (4, -1);

    \draw[thick] (0, 0) -- (5, 0) -- (5, -1) -- (4, -1) -- (4, -2) -- (3, -2) -- (3, -3) -- (2, -3) -- (2,-4) -- (1,-4) -- (1, -5) -- (0, -5) -- (0,0);
\end{tikzpicture}
\caption*{\textit{Step 3:} We consider a tiling by $n - 1$ rectangles.}
\label{fig:step3}
\end{minipage}
\hfill
\begin{minipage}{0.45\textwidth}
    \begin{tikzpicture}[scale=1.2]
    \fill[\colorA, nearly transparent] (0, 0) -- (2, 0) -- (2, -4) -- (0, -4);
    \fill[\colorB, nearly transparent] (2,-1) -- (2, -3) -- (3, -3) -- (3, -1);
    \fill[\colorC, nearly transparent] (2, 0) -- (5, 0) -- (5, -1) -- (2, -1);
    \fill[\colorD, nearly transparent] (3, -1) -- (4, -1) -- (4, -2) -- (3, -2);
     \fill[\colorE, nearly transparent] (0, -4) -- (1, -4) -- (1, -5) -- (0, -5);

    \draw (0.35, -0.2) node {\scalebox{0.8}{$(1,6)$}};
    \draw (1.35, -0.2) node {\scalebox{0.8}{$(1,5)$}};
    \draw (2.35, -0.2) node {\scalebox{0.8}{$(1,4)$}};
     \draw (3.35, -0.2) node {\scalebox{0.8}{$(1,3)$}};
    \draw (4.35, -0.2) node {\scalebox{0.8}{$(1,2)$}};
    \draw (0.35, -1.2) node {\scalebox{0.8}{$(2,6)$}};
    \draw (1.35, -1.2) node {\scalebox{0.8}{$(2,5)$}};
    \draw (2.35, -1.2) node {\scalebox{0.8}{$(2,4)$}};
    \draw (3.35, -1.2) node {\scalebox{0.8}{$(2,3)$}};
    \draw (0.35, -2.2) node {\scalebox{0.8}{$(3,6)$}};
    \draw (1.35, -2.2) node {\scalebox{0.8}{$(3,5)$}};
    \draw (2.35, -2.2) node {\scalebox{0.8}{$(3,4)$}};
    \draw (0.35, -3.2) node {\scalebox{0.8}{$(4,6)$}};
    \draw (1.35, -3.2) node {\scalebox{0.8}{$(4,5)$}};
    \draw (0.35, -4.2) node {\scalebox{0.8}{$(5,6)$}}; \draw (0.5, -0.5) node {1};
    \draw (1.5, -0.5) node {0};
    \draw (2.5, -0.5) node {0};
    \draw (3.5, -0.5) node {0};
    \draw (4.5, -0.5) node {0};
    \draw (0.5, -1.5) node {1};
    \draw (1.5, -1.5) node {1};
    \draw (2.5, -1.5) node {0};
    \draw (3.5, -1.5) node {1};
    \draw (0.5, -2.5) node {1};
    \draw (1.5, -2.5) node {0};
    \draw (2.5, -2.5) node {0};
    \draw (0.5, -3.5) node {1};
    \draw (1.5, -3.5) node {1};
    \draw (0.5, -4.5) node {1};

    \draw (5.25, -1.25) node {0};
    \draw (4.25, -2.25) node {1};
    \draw (3.25, -3.25) node {0};
    \draw (2.25, -4.25) node {6};
    \draw (1.25, -5.25) node {1};
    
    \draw[thick] (0, -1) -- (4, -1);
    \draw[thick] (0, -2) -- (3, -2);
    \draw[thick] (0, -3) -- (2, -3);
    \draw[thick] (0, -4) -- (1, -4);
    
    \draw[thick] (1, 0) -- (1, -4);
    \draw[thick] (2, 0) -- (2, -3);
    \draw[thick] (3, 0) -- (3, -2);
    \draw[thick] (4, 0) -- (4, -1);

    \draw[thick] (0, 0) -- (5, 0) -- (5, -1) -- (4, -1) -- (4, -2) -- (3, -2) -- (3, -3) -- (2, -3) -- (2,-4) -- (1,-4) -- (1, -5) -- (0, -5) -- (0,0);
\end{tikzpicture}
\caption*{\textit{Step 4:} $\Newton(D^{253641})$ has a vertex at $(0,1,0,6,1)$. }
\label{fig:step5}
\end{minipage}
\caption{$\Newton(D^{254361})$ has a vertex at $(0,1,0,6,1)$.}
\label{fig:01061}
\end{figure}
\end{example}
\begin{example}
    \Cref{fig:partition-tiling} illustrates the procedure for obtaining all vertices of $\Newton(D^{4213})$.
 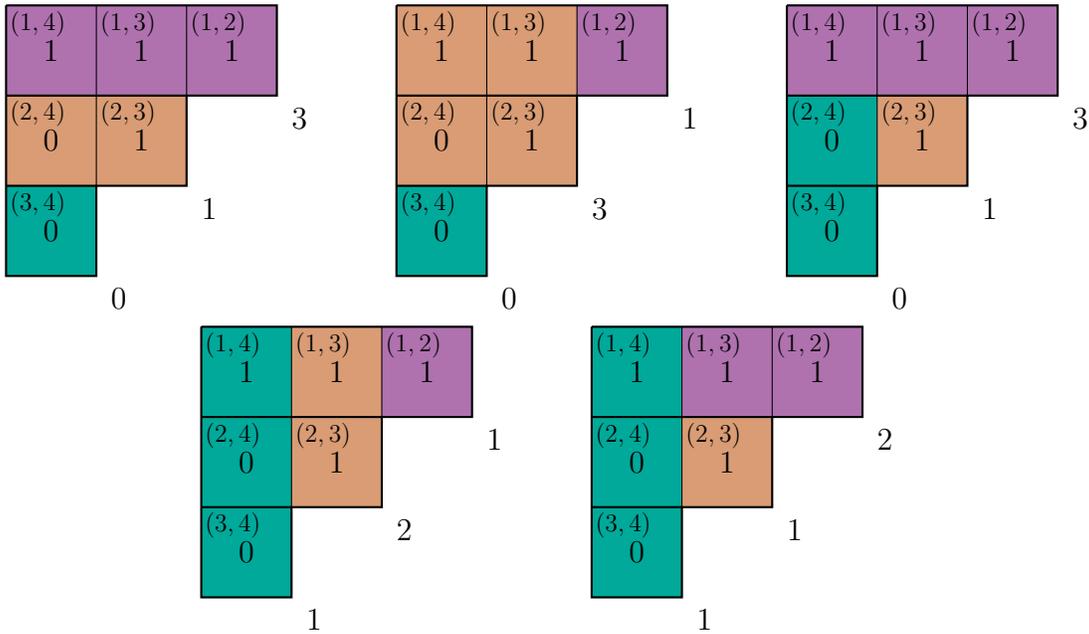
\begin{figure}[H]
       \centering
       \begin{tikzpicture}[scale=1.2]
    \fill[\colorA, nearly transparent] (0, 0) -- (3, 0) -- (3, -1) -- (0, -1);
    \fill[\colorB, nearly transparent] (0, -1) -- (2, -1) -- (2, -2) -- (0, -2);
    \fill[\colorC, nearly transparent] (0, -2) -- (1, -2) -- (1, -3) -- (0, -3);

    \draw (0.35, -0.2) node {\scalebox{0.8}{$(1,4)$}};
    \draw (1.35, -0.2) node {\scalebox{0.8}{$(1,3)$}};
    \draw (2.35, -0.2) node {\scalebox{0.8}{$(1,2)$}};
    \draw (0.35, -1.2) node {\scalebox{0.8}{$(2,4)$}};
    \draw (1.35, -1.2) node {\scalebox{0.8}{$(2,3)$}};
    \draw (0.35, -2.2) node {\scalebox{0.8}{$(3,4)$}};

    \draw (0.5, -0.5) node {1};
    \draw (1.5, -0.5) node {1};
    \draw (2.5, -0.5) node {1};
    \draw (0.5, -1.5) node {0};
    \draw (1.5, -1.5) node {1};
    \draw (0.5, -2.5) node {0};

    \draw (3.25, -1.25) node {3};
    \draw (2.25, -2.25) node {1};
    \draw (1.25, -3.25) node {0};
    
    \draw (0, -1) -- (2, -1);
    \draw (0, -2) -- (1, -2);
    \draw (1, 0) -- (1, -2);
    \draw (2, 0) -- (2, -1);

    \draw[thick] (0, -1) -- (2, -1);
    \draw[thick] (0, -2) -- (1, -2);

    \draw[thick] (0, 0) -- (3, 0) -- (3, -1) -- (2, -1) -- (2, -2) -- (1, -2) -- (1, -3) -- (0, -3) -- (0, 0);
\end{tikzpicture}
\qquad
\begin{tikzpicture}[scale=1.2]
    \fill[\colorA, nearly transparent] (2, 0) -- (3, 0) -- (3, -1) -- (2, -1);
    \fill[\colorB, nearly transparent] (0, 0) -- (2, 0) -- (2, -2) -- (0, -2);
    \fill[\colorC, nearly transparent] (0, -2) -- (1, -2) -- (1, -3) -- (0, -3);

    \draw (0.35, -0.2) node {\scalebox{0.8}{$(1,4)$}};
    \draw (1.35, -0.2) node {\scalebox{0.8}{$(1,3)$}};
    \draw (2.35, -0.2) node {\scalebox{0.8}{$(1,2)$}};
    \draw (0.35, -1.2) node {\scalebox{0.8}{$(2,4)$}};
    \draw (1.35, -1.2) node {\scalebox{0.8}{$(2,3)$}};
    \draw (0.35, -2.2) node {\scalebox{0.8}{$(3,4)$}};

    \draw (0.5, -0.5) node {1};
    \draw (1.5, -0.5) node {1};
    \draw (2.5, -0.5) node {1};
    \draw (0.5, -1.5) node {0};
    \draw (1.5, -1.5) node {1};
    \draw (0.5, -2.5) node {0};

    \draw (3.25, -1.25) node {1};
    \draw (2.25, -2.25) node {3};
    \draw (1.25, -3.25) node {0};
    
    \draw (0, -1) -- (2, -1);
    \draw (0, -2) -- (1, -2);
    \draw (1, 0) -- (1, -2);
    \draw (2, 0) -- (2, -1);

    \draw[thick] (0, -1) -- (2, -1);
    \draw[thick] (0, -2) -- (1, -2);

    \draw[thick] (0, 0) -- (3, 0) -- (3, -1) -- (2, -1) -- (2, -2) -- (1, -2) -- (1, -3) -- (0, -3) -- (0, 0);
\end{tikzpicture}
\qquad
\begin{tikzpicture}[scale=1.2]
    \fill[\colorA, nearly transparent] (0, 0) -- (3, 0) -- (3, -1) -- (0, -1);
    \fill[\colorB, nearly transparent] (1, -1) -- (2, -1) -- (2, -2) -- (1, -2);
    \fill[\colorC, nearly transparent] (0, -1) -- (1, -1) -- (1, -3) -- (0, -3);

    \draw (0.35, -0.2) node {\scalebox{0.8}{$(1,4)$}};
    \draw (1.35, -0.2) node {\scalebox{0.8}{$(1,3)$}};
    \draw (2.35, -0.2) node {\scalebox{0.8}{$(1,2)$}};
    \draw (0.35, -1.2) node {\scalebox{0.8}{$(2,4)$}};
    \draw (1.35, -1.2) node {\scalebox{0.8}{$(2,3)$}};
    \draw (0.35, -2.2) node {\scalebox{0.8}{$(3,4)$}};

    \draw (0.5, -0.5) node {1};
    \draw (1.5, -0.5) node {1};
    \draw (2.5, -0.5) node {1};
    \draw (0.5, -1.5) node {0};
    \draw (1.5, -1.5) node {1};
    \draw (0.5, -2.5) node {0};

    \draw (3.25, -1.25) node {3};
    \draw (2.25, -2.25) node {1};
    \draw (1.25, -3.25) node {0};
    
    \draw (0, -1) -- (2, -1);
    \draw (0, -2) -- (1, -2);
    \draw (1, 0) -- (1, -2);
    \draw (2, 0) -- (2, -1);

    \draw[thick] (0, -1) -- (2, -1);
    \draw[thick] (0, -2) -- (1, -2);

    \draw[thick] (0, 0) -- (3, 0) -- (3, -1) -- (2, -1) -- (2, -2) -- (1, -2) -- (1, -3) -- (0, -3) -- (0, 0);
\end{tikzpicture}

\begin{tikzpicture}[scale=1.2]
    \fill[\colorA, nearly transparent] (2, 0) -- (3, 0) -- (3, -1) -- (2, -1);
    \fill[\colorB, nearly transparent] (1, 0) -- (2, 0) -- (2, -2) -- (1, -2);
    \fill[\colorC, nearly transparent] (0, 0) -- (1, 0) -- (1, -3) -- (0, -3);

    \draw (0.35, -0.2) node {\scalebox{0.8}{$(1,4)$}};
    \draw (1.35, -0.2) node {\scalebox{0.8}{$(1,3)$}};
    \draw (2.35, -0.2) node {\scalebox{0.8}{$(1,2)$}};
    \draw (0.35, -1.2) node {\scalebox{0.8}{$(2,4)$}};
    \draw (1.35, -1.2) node {\scalebox{0.8}{$(2,3)$}};
    \draw (0.35, -2.2) node {\scalebox{0.8}{$(3,4)$}};

    \draw (0.5, -0.5) node {1};
    \draw (1.5, -0.5) node {1};
    \draw (2.5, -0.5) node {1};
    \draw (0.5, -1.5) node {0};
    \draw (1.5, -1.5) node {1};
    \draw (0.5, -2.5) node {0};

    \draw (3.25, -1.25) node {1};
    \draw (2.25, -2.25) node {2};
    \draw (1.25, -3.25) node {1};
    
    \draw (0, -1) -- (2, -1);
    \draw (0, -2) -- (1, -2);
    \draw (1, 0) -- (1, -2);
    \draw (2, 0) -- (2, -1);

    \draw[thick] (0, -1) -- (2, -1);
    \draw[thick] (0, -2) -- (1, -2);

    \draw[thick] (0, 0) -- (3, 0) -- (3, -1) -- (2, -1) -- (2, -2) -- (1, -2) -- (1, -3) -- (0, -3) -- (0, 0);
\end{tikzpicture}
\qquad
\begin{tikzpicture}[scale=1.2]
    \fill[\colorA, nearly transparent] (1, 0) -- (3, 0) -- (3, -1) -- (1, -1);
    \fill[\colorB, nearly transparent] (1, -1) -- (2, -1) -- (2, -2) -- (1, -2);
    \fill[\colorC, nearly transparent] (0, 0) -- (1, 0) -- (1, -3) -- (0, -3);

    \draw (0.35, -0.2) node {\scalebox{0.8}{$(1,4)$}};
    \draw (1.35, -0.2) node {\scalebox{0.8}{$(1,3)$}};
    \draw (2.35, -0.2) node {\scalebox{0.8}{$(1,2)$}};
    \draw (0.35, -1.2) node {\scalebox{0.8}{$(2,4)$}};
    \draw (1.35, -1.2) node {\scalebox{0.8}{$(2,3)$}};
    \draw (0.35, -2.2) node {\scalebox{0.8}{$(3,4)$}};

    \draw (0.5, -0.5) node {1};
    \draw (1.5, -0.5) node {1};
    \draw (2.5, -0.5) node {1};
    \draw (0.5, -1.5) node {0};
    \draw (1.5, -1.5) node {1};
    \draw (0.5, -2.5) node {0};

    \draw (3.25, -1.25) node {2};
    \draw (2.25, -2.25) node {1};
    \draw (1.25, -3.25) node {1};
    
    \draw (0, -1) -- (2, -1);
    \draw (0, -2) -- (1, -2);
    \draw (1, 0) -- (1, -2);
    \draw (2, 0) -- (2, -1);

    \draw[thick] (0, -1) -- (2, -1);
    \draw[thick] (0, -2) -- (1, -2);

    \draw[thick] (0, 0) -- (3, 0) -- (3, -1) -- (2, -1) -- (2, -2) -- (1, -2) -- (1, -3) -- (0, -3) -- (0, 0);
\end{tikzpicture}

\caption{$\Newton(D^{4213})$ has vertices $(3,1,0), (1,3,0), (1,2,1), (2,1,1)$.}
       \label{fig:partition-tiling}
   \end{figure}
\end{example}

\section{Future Directions}
We hope to generalize our results to Postnikov--Stanley polynomials. 
\begin{conjecture}
    Postnikov--Stanley polynomials are M-convex.
\end{conjecture}

We are also interested in a characterization of when $D_u^w$ has SCNP. The following pattern-containment conjecture has been verified using SageMath \cite{sagemath} for $S_n$ with $n\le 6$.
\begin{conjecture}
    For a permutation $u\in S_n$, there exists a permutation $w\in S_n$ such that $D_{u}^w$ does not have SCNP if and only if $u$ contains 1324. Analogously, for a permutation $w\in S_n$, there exists a permutation $u\in S_n$ such that $D_{u}^w$ does not have SCNP if and only if $w$ contains 4231. 
\end{conjecture}

\bibliographystyle{plain} 
\bibliography{main}

\end{document}